\documentclass{article}

\usepackage{amsmath,amssymb,amsthm}
\usepackage{cite}
\usepackage{url}

\usepackage{tikz}

\usepackage{array, multirow}

\title{On the next-nearest DLA model}

\date{}

\newtheorem{theorem}{Theorem}
\newtheorem{proposition}{Proposition}
\newtheorem{lemma}{Lemma}
\newtheorem{corollary}{Corollary}

\newtheorem{construction}{Construction}

\theoremstyle{definition}
\newtheorem{remark}{Remark}

\newtheorem{example}{Example}

\newcolumntype{M}[1]{>{\centering\arraybackslash}m{#1}}

\begin{document}

\begin{center}
\textbf{\large{ON THE GROWTH OF A BALLISTIC\\ DEPOSITION MODEL ON FINITE GRAPHS}}~\\~\\
{GEORG BRAUN}\\~\\
\end{center}

\begin{abstract}
\noindent We revisit a ballistic deposition process introduced by Atar et al. in \cite{AAK01}. Let $\mathcal{G}=(V,E)$ be a finite connected graph. We choose independently and uniformly vertices in $\mathcal{G}$. If a vertex $x$ is chosen and the previous height configuration is given by $h=(h_y)_{y \in V} \in \mathbb{N}_0^V$, the height $h_x$ is replaced by
\[
\tilde{h}_x := 1 + \max_{y \sim x} h_y.
\]
We study asymptotic properties of this growth model. We determine the asymptotic growth parameter $\gamma(\mathcal{G} )$ for some graphs and prove a central limit theorem for the fluctuations around $\gamma ( \mathcal{G})$. We also give a new graph-theoretic interpretation of an inequality obtained in \cite{AAK01}.
\end{abstract}

\section{Introduction}

Let us start with an informal description of our random growth model.\\

\textit{In a city there is an exclusive group of skyscraper owners. Once in a while, an owner decides to heighten his building until it is strictly higher than the skyscrapers of the group members he disrespects. If his building already achieves this, it will be raised by only one floor. How fast will the skyscrapers grow?}\\

Ballistic growth models are typically studied on infinite graphs, when they are believed to belong to the KPZ universality class and in two dimensions exhibit fluctuations of the order $t^{1/3}$ as the time $t$ goes to infinity (compare e.g. \cite{Qua12} and \cite{Cor16} ). However, exact results of this kind have been proven only for a few specific cases in the KPZ universe (see for example \cite{BP14} and \cite{BG16}).

In this article we restrict our attention to the case of finite underlying graphs and study the asymptotic growth of a specific deposition model. We give formulas for the asymptotic growth rate in some explicitly solvable cases and present some further asymptotic results. We also prove a classical central limit theorem for our growth model, which holds for arbitrary graphs. Finally, we give an upper bound for the asymptotic growth parameter in terms of the maximal degree of the underlying graph. This inequality is based on the methods used in \cite{AAK01}.

Let $\mathcal{G}$ be a connected undirected graph with a finite non-empty vertex set $V$ and edge set $E \subseteq \{ \{ x,y \}~|~x,y \in V,~ x \neq y \}$. For a vertex $x \in V$ we define its (closed) neighbourhood by

\[ [x] := \{ x \} \cup \left\{ y \in V~\vert~ \{x,y\} \in E \right\}.    \]

\noindent As time goes by, we successively choose independently and uniformly vertices in the graph $\mathcal{G}$. If a vertex $x \in V$ is chosen and the previous height of the process is given by $(h_y)_{y \in V} \in \mathbb{N}_0^V$, the height $h_x$ will be replaced by

\begin{equation}
\tilde{h}_x := 1 + \max\limits_{y \in [x] }~h_y .
\end{equation}

\noindent This rule defines the so-called next nearest neighbour ballistic deposition model. We distinguish between the following two closely related versions of this process.

On the one hand, we can let the time evolve in discrete steps $n=1,2,\ldots$ and always choose exactly one vertex at these time points. The height of a vertex $x \in V$ after $n$ steps will then be denoted by $H_{x,n}$. Our process is then given by $H:=(H_n)_{n \geq 0}$, where $H_n := (H_{x,n})_{x \in V}$.

On the other hand, we may choose a family $(\xi_x)_{x \in V}$ of independent Poisson processes and change the height in a vertex $x \in V$ at time $t \in (0,\infty)$ if and only if the corresponding Poisson process $\xi_x$ jumps at time $t$. Unless explicitly stated otherwise, we will assume that all Poisson processes have unit intensity. We will write $\tilde{H}_{x,t}$ for the height of a vertex $x \in V$ at time $t \in [0,\infty)$ and set $\tilde{H}_t := ( \tilde{H}_{x,t})_{x \in V}$ as well as $\tilde{H} := (\tilde{H}_t)_{t \geq 0}$.

Note that both $H$ and $\tilde{H}$ are time-homogeneous Markov processes. Usually we will assume the initial condition $H_{x,0} = \tilde{H}_{x,0} = 0$ for all $x \in V$, which ensures that both Markov processes have the same state space.

In \cite{AAK01} Atar, Athreya and Kang considered the specific case of a cyclic graph $\mathcal{G}= \mathcal{C}_n$, which can be defined to have vertex set $ \{ 1, \ldots, n \}$ and edge set $\{ \{ 1,2 \}, \{2,3 \} , \ldots, \{ n-1, n\} \}$. Then, as explained in \cite{AAK01}, Kingman's subadditive ergodic theorem yields the existence of the almost sure limit

\begin{equation}
\gamma ( \mathcal{C}_n)   = \lim\limits_{t \rightarrow \infty} \frac{1}{t} \max\limits_{x \in V} \tilde{H}_{x,t} = \lim\limits_{t \rightarrow \infty} \frac{1}{t} \min\limits_{x \in V} \tilde{H}_{x,t} \in (0,\infty).  
\end{equation}

In fact, these arguments apply in the same way to a general graph $\mathcal{G}$ and hence we always define the growth parameter $\gamma ( \mathcal{G})$ by the right hand side of (2). The asymptotic growth of our time discrete model is related to $\gamma(\mathcal{G})$ via
\begin{align}
\gamma( \mathcal{G}) = \# V \lim\limits_{n \rightarrow \infty} \frac{1}{n} \max\limits_{x \in V} H_{x,n} = \# V \lim\limits_{n \rightarrow \infty} \frac{1}{n} \min\limits_{x \in V} H_{x,n},
\end{align}

\noindent where the limits again hold almost surely. For studying the parameter $\gamma ( \mathcal{G})$ of a given graph $\mathcal{G}$ we therefore can switch from continuous time to discrete time or vice versa, and this will turn out to be advantageous sometimes.

Let us now briefly summarise the relevant previous literature on our model. The main result of \cite{AAK01} is the inequality
\begin{align}
 3.21 < \gamma ( \mathcal{C}_n) < 5.35 \quad \textnormal{for~all~} n \geq 5. 
\end{align}
\noindent The authors of \cite{AAK01} also claimed that this inequality holds for $n=4$. However, as we will see in Section 4,
\[ \gamma ( \mathcal{C}_4 ) = 2 + \frac{2}{\sqrt{3}} \approx 3.1547.  \]
This reveals a minor calculation error in \cite{AAK01} for $n=4$. The statement (4) and its proof in \cite{AAK01} are correct, however.

In \cite{FFK11} Fleurke, Formentin and K{\"u}lske assumed that the vertices of the graph $\mathcal{G}$ are not chosen uniformly, but according to a fixed Markov chain with state space $V$. They proved the existence of the limit $\gamma ( \mathcal{G})$ in this more general setting, as well as a sub-Gaussian concentration inequality for the maximal height.

In \cite{MS12} Mountford and Sudbury studied homogeneous isotropic infinite graphs and related the growth parameter $\gamma ( \mathcal{G})$ to the roughness of the surface.

In \cite{MRR19} Mansour, Rastegar and Roitershtein discussed combinatorial problems related to our model in the case of $\mathcal{G} = \mathcal{C}_n$ and conjectured that

\[  \lim\limits_{n \rightarrow \infty} \gamma(\mathcal{C}_n) = 4.  \]

For convenience of the reader, we will briefly explain the structure of the present article and where to find which result. 

In Section 2 we briefly explain relevant graph theoretic concepts and introduce some notations.

Then, in Section 3, we mainly concentrate on the class of star graphs. We will also give an example of non-isomorphic graphs $\mathcal{G}$ and $\mathcal{H}$ with $\gamma ( \mathcal{G}) = \gamma ( \mathcal{H})$.

In Section 4 a rather simple probabilistic approach is used to compute the growth parameter in a specific setting.

Section 5 contains a central limit theorem for the fluctuations around $\gamma ( \mathcal{G})$, which is proved in a rather elementary way. We also briefly explain how one can deduce more information on these fluctuations.

In Section 6 we give an upper bound for the growth parameter $\gamma ( \mathcal{G})$ by using methods of spectral graph theory. This result is based on some modifications of the arguments used in \cite{AAK01}.

In Section 7 we briefly look at another growth model, which arises by slightly modifying our deposition rule (1).

\section{Graph-theoretic Preliminaries}

The degree of a vertex $x \in V$ is $\deg (x) := \# [x] - 1$ and the maximal degree in $\mathcal{G}$ is $\Delta \mathcal{G} := \max_{x \in V} \deg(x)$. 	The graph $\mathcal{G}$ will be called regular if $\deg(x)=\deg(y)$ for all $x$, $y \in V$. A vertex $x \in V$ is called dominant in the graph $\mathcal{G}$ if $[x] = V$. 

A path of length $n$ in $\mathcal{G}$ is a tuple $(x_1,\ldots,x_n) \in V^n$ with $\{ x_i, x_{i+1} \} \in E$ for all $i=1,\ldots,n-1$. If in addition $x_1 = x_n$, $n \geq 2$ and $x_i \neq x_j$ for all $i,j=1,\ldots,n-1$ with $i \neq j$, we will call $(x_1,\ldots,x_n)$ a cycle in $\mathcal{G}$. The length of the smallest cycle of a graph $\mathcal{G}$ will be denoted by girth($\mathcal{G}$). If there is no cycle in $\mathcal{G}$, we set girth($\mathcal{G}):=\infty$. We define $d_\mathcal{G}(x,x):=0$ for all $x \in V$, and for vertices $x \neq y$ we define $d_\mathcal{G}(x,y)$ to be the length of the smallest path from $x$ to $y$. Note that $d_\mathcal{G}$ is a metric on $V$. A permutation $(x_1,\ldots,x_{\# V})$ of $V$ will be called non-decreasing, if the function $k \mapsto d_\mathcal{G}(x_1,x_k)$ is non-decreasing.

For a graph $\mathcal{G}$ we denote by $A(\mathcal{G})$ a $\# V \times \# V$ adjacency matrix of $\mathcal{G}$. For $(x,y) \in V^2$ the corresponding entry of $A(\mathcal{G})$ is one, if $\{ x,y \} \in E$, and zero otherwise. Recall that more generally, for fixed $n \geq 1$,  the matrix entries of $A(\mathcal{G})^n$ count the number of paths between two vertices.

Given two graphs $\mathcal{G} = (V,E)$, $\mathcal{G}' = (V', E')$ we say that $\mathcal{G}$ is a subgraph of $\mathcal{G}'$ if $V \subseteq V'$ and $E \subseteq E'$.

For $n \geq 1$ we will write $\mathcal{S}_n$ to denote a star graph with $n$ vertices. Formally, one can choose $\{ 1, \ldots, n \}$ and $\{ \{ 1, 2 \}, \ldots, \{ 1, n \} \}$ as vertex set respectively edge set. We further denote by $\mathcal{K}_n$ a complete graph with $n$ vertices. For even $n \geq 2$ we denote by $\mathcal{R}_n$ a graph, which consists of $n$ vertices and is regular with $\Delta \mathcal{R}_n = n-2$. Note that $\mathcal{R}_n$ is unique up to an isomorphism. Formally, one might choose $\{ 1, \ldots, n \}$ and $\{ \{ 1,2 \}, \ldots , \{ 1, n-1 \}, \{2,3 \}, \ldots, \{ 2, n-2 \}, \{ 2, n \}, \{3,4 \}, \ldots \}$ as vertex set respectively edge set.

Note that $\gamma ( \mathcal{K}_n) = n$ for all $n \geq 1$. 
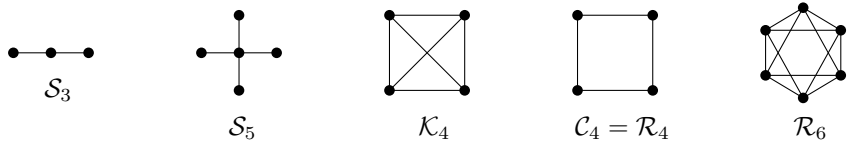
\begin{figure}[!h]
\begin{center}
\begin{tikzpicture}[baseline=0]
    \fill (0,0) circle (2pt);
    \fill (0.5,0) circle (2pt);
    \fill (-0.5,0) circle (2pt);
    \draw (0,0)--(0.5,0);
    \draw (0,0)--(-0.5,0);
    
    \node at (0.1,-0.5) {$\mathcal{S}_3$};
    
	\begin{scope}[xshift=2.5cm]
    \fill (0,0) circle (2pt);
    \fill (0.5,0) circle (2pt);
    \fill (-0.5,0) circle (2pt);
    \fill (0,0.5) circle (2pt);
    \fill (0,-0.5) circle (2pt);
    \draw (0,0)--(0.5,0);
    \draw (0,0)--(-0.5,0);
    \draw (0,0)--(0,0.5);
    \draw (0,0)--(0,-0.5);    
    \node at (0.05,-1) {$\mathcal{S}_5$};
	\end{scope}
	
	\begin{scope}[xshift=5cm]
	\fill (-0.5,-0.5) circle (2pt);
	\fill (-0.5,0.5) circle (2pt);
	\fill (0.5,-0.5) circle (2pt);
	\fill (0.5,0.5) circle (2pt);
	\draw (0.5,0.5)--(-0.5,0.5)--(-0.5,-0.5)--(0.5,-0.5)--(0.5,0.5);
	\draw (0.5,0.5)--(-0.5,-0.5);
	\draw (-0.5,0.5)--(0.5,-0.5);
	\node at (0.1,-1) {$\mathcal{K}_4$};
	\end{scope}
	
	\begin{scope}[xshift=7.5cm]
	\fill (-0.5,-0.5) circle (2pt);
	\fill (-0.5,0.5) circle (2pt);
	\fill (0.5,-0.5) circle (2pt);
	\fill (0.5,0.5) circle (2pt);
	\draw (0.5,0.5)--(-0.5,0.5)--(-0.5,-0.5)--(0.5,-0.5)--(0.5,0.5);
	\node at (0.1,-1) {$\mathcal{C}_4 = \mathcal{R}_4$};
	\end{scope}
	
	\begin{scope}[xshift=10cm]
	\fill (-0.5,-0.3) circle (2pt);
	\fill (-0.5,0.3) circle (2pt);	
	\fill (0,0.6) circle (2pt);	
	\fill (0,-0.6) circle (2pt);
	\fill (0.5,-0.3) circle (2pt);
	\fill (0.5,0.3) circle (2pt);
	\draw (-0.5,-0.3)--(0,-0.6)--(0.5,-0.3)--(0.5,0.3)--(0,0.6)--(-0.5,0.3)--(-0.5,-0.3);	
	\draw (-0.5,-0.3)--(0.5,-0.3)--(0,0.6)--(-0.5,-0.3);
	\draw (-0.5,0.3)--(0.5,0.3)--(0,-0.6)--(-0.5,0.3);
	\node at (0.1,-1) {$\mathcal{R}_6$};
	\end{scope} 
	\end{tikzpicture}
\end{center}
\caption{Some of the graphs we will study.}
\end{figure}

By definition of our deposition model each vertex $x$ interacts with the growth of the process only by its (closed) neighbourhood $[x]$. Therefore vertices $x$, $y \in V$ will be called equivalent, if $[x]=[y]$. The graph, which arises from $\mathcal{G}$ by identifying all equivalent vertices, will be denoted by $\hat{\mathcal{G}}$ and called an irreducible graph. Note that the (asymptotic) growth in our model does not change, if we replace $\mathcal{G}$ by $\hat{\mathcal{G}}$ and modify the intensity of the underlying Poisson processes accordingly. More precisely, the intensity of the Poisson process associated to a vertex $\hat{x}$ in $\hat{\mathcal{G}}$ has to equal the number of vertices $x$ in $\mathcal{G}$ which have been contracted into $\hat{x}$.

Note that this transformation can also be applied in the reversed way. Let a graph with positive integer intensities for all vertices be given. Then we can stepwise choose the vertices with not unit intensity, define a new adjacent vertex with unit intensity and the same closed neighbourhood, and reduce the intensity of the originally chosen vertex by one. We will use the term vertex cloning for this procedure. Again, note that the order, in which the vertices are chosen, does not affect the resulting graph up to an isomorphism.

\begin{example}
For the butterfly graph $\mathcal{B}$ we have the identification

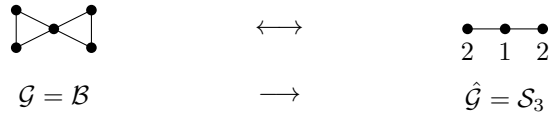
\begin{center}
\begin{tikzpicture}[baseline=0]
    \fill (0,0) circle (2pt);
    \fill (0.5,0.25) circle (2pt);
    \fill (0.5,-0.25) circle (2pt);
    \fill (-0.5,-0.25) circle (2pt);
    \fill (-0.5,0.25) circle (2pt);
    \draw (0,0)--(0.5,0.25);
    \draw (0,0)--(0.5,-0.25);
    \draw (0,0)--(-0.5,0.25);
    \draw (0,0)--(-0.5,-0.25);
    \draw (0.5,-0.25)--(0.5,0.25);
    \draw (-0.5,-0.25)--(-0.5,0.25);
    
    \node at (0,-0.9) {$\mathcal{G} = \mathcal{B}$};
    
	\begin{scope}[xshift=3cm]
    \node at (0,0) {$\longleftrightarrow$};
    \node at (0,-0.9) {$\longrightarrow$};
	\end{scope}    
    
    \begin{scope}[xshift=6cm]
    \fill (0,0) circle (2pt);
    \fill (0.5,0) circle (2pt);
    \fill (-0.5,0) circle (2pt);
    \draw (0,0)--(0.5,0);
    \draw (0,0)--(-0.5,0);
    \node at (0.0,-0.3) {$1$};
    \node at (0.5,-0.3) {$2$};
    \node at (-0.5,-0.3) {$2$};
    \node at (0,-0.9) { $ \hat{\mathcal{G}} = \mathcal{S}_3$ };
	\end{scope}    
	\end{tikzpicture}
\end{center}
\end{example}

\noindent For studying the growth in our ballistic deposition model, the following three settings are essentially the same.

\begin{tabular}{c l}
(i)&Arbitrary graphs $\mathcal{G}$ with unit intensities.\\
(ii)&Arbitrary graphs $\mathcal{G}$ with positive integer intensities.\\
(iii)&Irreducible graphs $\hat{\mathcal{G}}$ with positive integer intensities.\\ \\
\end{tabular}

\noindent Observe that when working in the setting (ii), the asymptotic growth parameter changes linearly if we multiply all intensities by a fixed constant. Translating this into our original setting (i) therefore yields the following construction.

\begin{construction}
Let $\mathcal{G}$ be a graph and $n \in \mathbb{N}$. Then there is a graph $\mathcal{H}$ with
\[ \gamma ( \mathcal{H}) = n~\gamma ( \mathcal{G}). \]
Such a graph $\mathcal{H}$ can be obtained by cloning each vertex of $\mathcal{G}$ exactly $n-1$ times.
\end{construction}

\begin{example}
Let us illustrate Construction 1 for $n=2$ and $\mathcal{G} = \mathcal{S}_3$, when we start with cloning the dominant vertex.

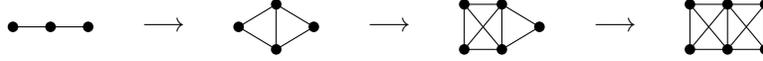
\begin{center}
\begin{tikzpicture}

	\fill (-5,0) circle (2pt);
    \fill (-4.5,0) circle (2pt);
    \fill (-4,0) circle (2pt);
   	\draw (-5,0)--(-4.5,0)--(-4,0);
    
    \node at (-3,0) {$\longrightarrow$};    
    
	\fill (-2,0) circle (2pt);
	\fill (-1.5,0.3) circle (2pt);
	\fill (-1.5,-0.3) circle (2pt);
	\fill (-1,0) circle (2pt);
	\draw (-2,0)--(-1.5,0.3)--(-1,0)--(-1.5,-0.3)--(-2,0);
	\draw (-1.5,0.3)--(-1.5,-0.3);

	\node at (0,0) {$\longrightarrow$};

	\fill (1,-0.3) circle (2pt);
	\fill (1,0.3) circle (2pt);
	\fill (1.5,0.3) circle (2pt);
	\fill (1.5,-0.3) circle (2pt);
	\fill (2,0) circle (2pt);
	\draw (1,0.3)--(1.5,0.3)--(2,0);
	\draw (1,-0.3)--(1.5,-0.3)--(2,0);
	\draw (1,0.3)--(1,-0.3);
	\draw (1.5,0.3)--(1.5,-0.3);
	\draw (1,0.3)--(1.5,-0.3);
	\draw (1,-0.3)--(1.5,0.3);		
	
	\node at (3,0) {$\longrightarrow$};
	
	\fill (4,-0.3) circle (2pt);
	\fill (4,0.3) circle (2pt);
	\fill (4.5,0.3) circle (2pt);
	\fill (4.5,-0.3) circle (2pt);
	\fill (5,0.3) circle (2pt);
	\fill (5,-0.3) circle (2pt);
	\draw (4,0.3)--(4.5,0.3)--(5,0.3);
	\draw (4,-0.3)--(4.5,-0.3)--(5,-0.3);
	\draw (4,0.3)--(4.5,-0.3)--(5,0.3);
	\draw (4,-0.3)--(4.5,0.3)--(5,-0.3);
	\draw (4,-0.3)--(4,0.3);
	\draw (4.5,-0.3)--(4.5,0.3);
	\draw (5,-0.3)--(5,0.3);
	
\end{tikzpicture}~\\
\end{center}

\noindent For the resulting graph $\mathcal{H}$ we know $\gamma ( \mathcal{H} ) = 2  \gamma ( \mathcal{S}_3)$.
\end{example}

\section{On the Sequence of Star Graphs}

Fix $n \geq 3$ and consider the case $\mathcal{G} = \mathcal{S}_n$. Then, by stopping our deposition process at the points of time, at which the height of the dominant vertex is increased, we obtain a process with i.i.d. increments. By recalling the equations (2) and (3) and applying the law of large numbers for both discrete and continuous time version, we obtain the formulas
\begin{align}
\gamma ( \mathcal{S}_n) = 1 + \frac{1}{n} \sum\limits_{k=1}^\infty \frac{a_{n-1,k}}{n^{k}}, \quad a_{n,k} := \sum\limits_{\substack{r_1,\ldots,r_{n} \geq 0 \\ r_1 + \ldots + r_{n} = k} }  \binom{n}{r_1, \ldots, r_{n}}~\max_{j=1,\ldots,n} r_j .
\end{align} \begin{align}
\gamma ( \mathcal{S}_n) = 1 + \mathbb{E} \left[ \max\limits_{j=1,\ldots,n-1} U_{j,W} \right],
\end{align}
where $W$ denotes an exponentially distributed random variable with mean $1$, $U_{j,\lambda}$ is Poisson distributed with mean $\lambda$ for all $j \in \mathbb{N}$ and $\lambda \in (0,\infty)$, and all random variables are assumed to be independent of each other. 

In fact, we can determine the exact value of $\gamma ( \mathcal{S}_3)$ by working directly with (5). Fix $k \geq 1$. Then we have 
\begin{align*}
a_{2,2k} &= \sum\limits_{l=0}^{2k} \binom{2k}{l} \max \{ l, 2k -l \} = 2 \sum\limits_{l=k+1}^{2k} \binom{2k}{l} l + k \binom{2k}{k}\\
&= 4k \sum\limits_{l=k}^{2k-1} \binom{2k-1}{l} + k \binom{2k}{k} = k 2^{2k} + k \binom{2k}{k}.
\end{align*}
In the same way we obtain 
\[ a_{2,2k+1} = (2k+1)  2^{2k} + (2k+1)  \binom{2k}{k}. \]
These two formulas allow us to directly verify the recurrence relation 
\begin{equation*}
a_{2,k} = 2~\frac{k}{k-1}~a_{2,k-1} +4~\frac{k-3}{k-2}~a_{2,k-2} -8~a_{2,k-3}  \quad \textnormal{for~all~} k \geq 3.
\end{equation*}
By neglecting the last term in this recursion we easily see that for a $s \in (0,1)$ small enough the generating function $g(s):=\sum_{k=1}^\infty \frac{a_{2,k}}{k} s^k$ is finite. Hence, for all $s \in (0,1)$ small enough, the recurrence relation implies
\[  \left( 2s+ 1 \right) s \left( \left( 4 s^2 -1 \right) g'(s) + 2 g(s) \right) + 2 s (s-1) = 0.  \]
Using the initial condition $g(0)=0$ we obtain for all $s \in (0,1)$ small enough 
\[ g(s) = \frac{  4s - 1 + \sqrt{1 - 4s^2}}{2 - 4 s}. \]
By monotone convergence we conclude that this formula holds for all $s \in [0,1/2)$ and therefore equation (5) yields 
\[  \gamma ( \mathcal{S}_3) = 1 + \frac{1}{9} \cdot g' \left( \frac{1}{3} \right) =  2 + \frac{1}{\sqrt{5}} . \]

\begin{remark}
The sequence $(a_{2,k})_{k \geq 1}$ is mentioned in the OEIS under A230137. 
\end{remark}

\begin{remark}
The series representation in (5) is hard to work with in general, but at least allows rather precise calculations. We obtain, for example,
\[  \gamma ( \mathcal{S}_4) = 2.72446357391224888 \ldots . \]
We could not find an integer coefficient polynomial, which might have this value as a root. Hence we conjecture that $\gamma ( \mathcal{S}_4)$ is transcendental.
\end{remark}

\begin{proposition}
There are non-isomorphic graphs $\mathcal{G}$ and $\mathcal{H}$ with $\gamma ( \mathcal{G}) = \gamma ( \mathcal{H})$.
\end{proposition}

\begin{proof}
For the butterfly graph $\mathcal{B}$ we find by a similar calculation
\[ \gamma ( \mathcal{B}) = 1 + \frac{2}{25} \cdot g' \left( \frac{2}{5} \right) = \frac{11}{3}.  \]
By applying Construction 1 with $n=3$ to $\mathcal{B}$ we therefore obtain a graph $\mathcal{H}$ with $\gamma ( \mathcal{H} ) = 11 = \gamma ( \mathcal{K}_{11})$, which is clearly not isomorphic to $\mathcal{K}_{11}$. \qedhere \end{proof}

We have the following combinatorial interpretation of equation (5). Assume we have $n$ bins and $m$ balls. Then we throw the balls independently of each other in one uniformly chosen bin. Denote by $Z_{n,m}$ the number of balls in the maximally loaded box and let $Y_n$ be a random variable, which is independent of $(Z_{n,m})_{m \geq 1}$ and geometrically distributed with mean $n$. Then (5) reads as \medskip
\[ \gamma ( \mathcal{S}_n) = 1 + \mathbb{E} \left[ Z_{n-1,Y_{n-1}} \right]. \]
A classical result due to Gonnet, see \cite{Gon81}, states that for fixed $c \in (0,\infty)$ \medskip
\begin{equation}
\mathbb{E} \left[ Z_{n, \lfloor c n \rfloor} \right] \sim \Gamma^{-1}(n) \sim\frac{\log(n)}{\log(\log(n))} \quad \textnormal{for~} n \rightarrow \infty. 
\end{equation}
Here, and in the following, we will use the notation $a_n \sim b_n$ for $n \rightarrow \infty$ instead of $\lim_{n \rightarrow \infty} a_n/b_n = 1$.

An important tool in Gonnet's proof of (7) is the Poisson approximation. More precisely, it is verified in \cite{Gon81}, that for fixed $\lambda \in (0,\infty)$
\begin{align}
f_n(\lambda) := \mathbb{E} \left[ \max\limits_{j=1,\ldots,n} U_{j,\lambda} \right] \sim \frac{ \log(n)}{\log(\log(n))} \quad \textnormal{for~} n \rightarrow \infty. 
\end{align}

\begin{proposition}
\[ \gamma ( \mathcal{S}_n) \sim \frac{\log(n)}{\log(\log(n))} \quad \textnormal{as}~ n \rightarrow \infty.  \]
In particular $\gamma ( \mathcal{S}_n) \rightarrow \infty$ for $n \rightarrow \infty$.
\end{proposition}

\begin{proof}
The convolution property of the Poisson distribution implies that the functions $f_n = f_n ( \lambda)$, $n \geq 1$, are both monotone increasing and subadditive.

\noindent Fix $r \in (0,\infty)$ and recall equation (6). Then, due to monotonicity, we have 
\begin{align*}
\gamma ( \mathcal{S}_n) = 1 + \int_0^\infty e^{-\lambda} f_{n-1} ( \lambda)~\textnormal{d} \lambda \geq \int_r^\infty e^{-\lambda} f_{n-1} (r)~\textnormal{d} \lambda = e^{-r} f_{n-1} ( r).
\end{align*}
Now apply (8) and let $r \rightarrow 0$ to conclude
\[ \liminf\limits_{n \rightarrow \infty} \gamma ( \mathcal{S}_n) \frac{\log(\log(n))}{\log(n)} \geq 1.  \]
On the other hand, we have for fixed $r \in (0,\infty)$ the estimate
\begin{align*}
\gamma ( \mathcal{S}_n) &\leq 1 + \int_0^r e^{-\lambda} f_{n-1} (r)~\textnormal{d} \lambda + \int_r^\infty e^{-\lambda} f_{n-1} ( \lambda)~\textnormal{d} \lambda\\
&= 1 + \left( 1 - e^{-r} \right) f_{n-1} (r) + e^{-r} \int_0^\infty e^{-\lambda} f_{n-1} ( \lambda + r)~\textnormal{d} \lambda.
\end{align*}
Now, by using subadditivity of $f_{n-1}$, we obtain
\[
\gamma ( \mathcal{S}_n) \leq 1 + f_{n-1}(r) + e^{-r} \int_0^\infty e^{-\lambda} f_{n-1} (\lambda)~\textnormal{d}\lambda \leq 1 + f_{n-1}(r) + e^{-r} \gamma (\mathcal{S}_n).
\]
By applying (8) and letting $r \rightarrow \infty$ we therefore conclude
\[ \limsup\limits_{n \rightarrow \infty} \gamma ( \mathcal{S}_n) \frac{\log(\log(n))}{\log(n)} \in [0,1]. \qedhere  \] 
\end{proof}

\section{A more General Approach for Calculations}

\begin{theorem}
Let $N \geq 0$, $n \geq 1$ and $m \geq 2$. Assume $m$ is even and $N \geq 1$ if $m=2$. Let $\mathcal{G} = (V,E)$ be the graph, which arises from $\mathcal{R}_m$ by the following procedure. 

\begin{tabular}{c l}
\textnormal{(i)}&Clone each vertex of $\mathcal{R}_m$ exactly $n-1$ times.\\
\textnormal{(ii)}&Add $N$ new vertices $x_1,\ldots,x_N$ to $V$.\\
\textnormal{(iii)}&Add all edges of the form $\{ x_i,y \}$ with $y \in V \setminus \{ x_i \}$ to $E$.\\ \\
\end{tabular}

\noindent Note that $\# V = N+nm$. Set $\kappa:= \frac{\# V}{2n}$ and $\tau := ( \sqrt{\kappa^2 -1} - \kappa + 1  )^{-1}$. Then
\[ \gamma ( \mathcal{G} ) = \# V - \frac{n^2 m}{\# V} \tau \left\{ \tau + \frac{1}{2 \kappa} \right\}^{-1}  . \]
\end{theorem}

\begin{proof}
We define a stochastic process $(\Delta_n)_{n \geq 0}$ as follows. If a dominant vertex has the maximal height at time $n \geq 0$, we set $\Delta_n:=0$. Otherwise, there are at most two different vertices which share the maximal height. If there are two different vertices of maximal height, again set $\Delta_n := 0$. Otherwise, let $x$ be the unique vertex of maximal height and choose a vertex $y$, whose height is maximal under all vertices, which are not equivalent to $x$. Then set $\Delta_n := H_{x,n} - H_{y,n}$.

It is not hard to see that the process $(\Delta_n)_{n \geq 0}$ is a time-homogeneous Markov chain. Consider, for example, the case $\Delta_n = m$ for a $n \in \mathbb{N}$ and a $m \geq 3$. Denote by $x$ be the unique vertex of maximal height and by $y$ be the vertex, whose height is increased in the next step. Then we know $\Delta_{n+1} = 0$ if $y$ is a dominant vertex in $\mathcal{G}$. If $y$ is equivalent to $x$, we can conclude $\Delta_{n+1} = m+1$. We also know $\Delta_{n+1} = m-1$ if $y$ is a vertex, which is not connected to $x$. Finally, if $y$ is connected to $x$ but neither equivalent to $x$ nor dominant in $\mathcal{G}$, then $\Delta_{n+1} = 1$.

The transition probabilities of $(\Delta_n)_{n \geq 0}$ are illustrated in the following picture.

\noindent 
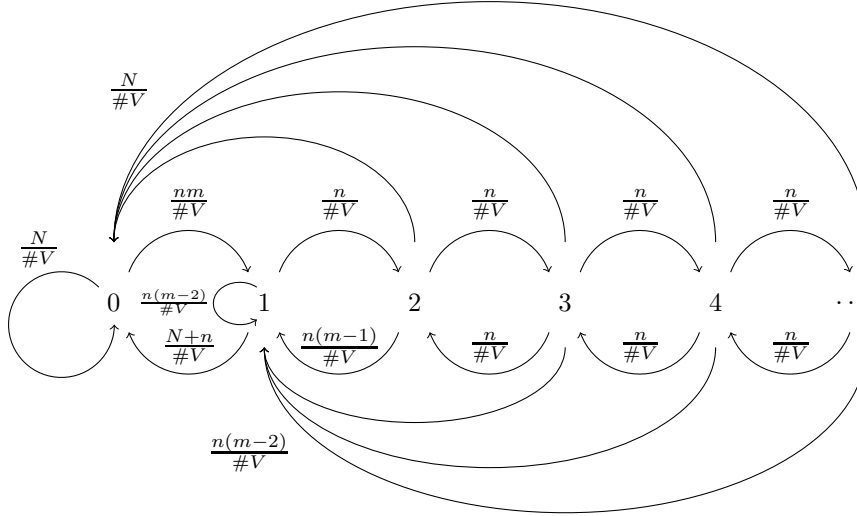
\begin{tikzpicture}[scale=2] 
\node (0) at (0,0) {$0$};
\node (1) at (1,0) {$1$};
\node (2) at (2,0) {$2$};
\node (3) at (3,0) {$3$};
\node (4) at (4,0) {$4$};
\node (...) at (4.9,0) {$\ldots$};
\draw [->] (0.1,0.2) arc (160:20:12pt);
\draw [->] (1.1,0.2) arc (160:20:12pt);
\draw [->] (2.1,0.2) arc (160:20:12pt);
\draw [->] (3.1,0.2) arc (160:20:12pt);
\draw [->] (4.1,0.2) arc (160:20:12pt);
\draw [<-] (0.1,-0.2) arc (200:340:12pt);
\draw [<-] (1.1,-0.2) arc (200:340:12pt);
\draw [<-] (2.1,-0.2) arc (200:340:12pt);
\draw [<-] (3.1,-0.2) arc (200:340:12pt);
\draw [<-] (4.1,-0.2) arc (200:340:12pt);
\draw [->] (-0.1,0.1) arc (45:360:10pt);
\draw [->] (0.95,0.1) arc(50:310:5pt and 4pt);

\node at (-0.5,0.35) {$\frac{N}{\# V}$};
\node at (0.5,0.65) {$\frac{nm}{\# V}$};
\node at (1.5,0.65) {$\frac{n}{\# V}$};
\node at (2.5,0.65) {$\frac{n}{\# V}$};
\node at (3.5,0.65) {$\frac{n}{\# V}$};
\node at (4.5,0.65) {$\frac{n}{\# V}$};
\node at (0.5,-0.3) {$\frac{N+n}{\# V}$};
\node at (0.4,0) {\tiny $\frac{n(m-2)}{\# V}$};
\node at (1.5,-0.3) {$\frac{n(m-1)}{\# V}$};
\node at (2.5,-0.3) {$\frac{n}{\# V}$};
\node at (3.5,-0.3) {$\frac{n}{\# V}$};
\node at (4.5,-0.3) {$\frac{n}{\# V}$};
\node at (0.1,1.4) {$\frac{N}{\# V}$};
\node at (0.9,-1.0) {$\frac{n(m-2)}{\# V}$};

\draw[<-] (0,0.4) arc(180:0:1cm and 0.7cm);
\draw[<-] (0,0.4) arc(180:0:1.5cm and 1cm);
\draw[<-] (0,0.4) arc(180:0:2cm and 1.3cm);
\draw[<-] (0,0.4) arc(180:0:2.5cm and 1.6cm);

\draw[<-] (1,-0.3) arc(180:360:1cm and 0.5cm);
\draw[<-] (1,-0.3) arc(180:360:1.5cm and 0.8cm);
\draw[<-] (1,-0.3) arc(180:360:2cm and 1.1cm);
\end{tikzpicture}~\\~\\

\noindent The chain $(\Delta_n)_{n \geq 0}$ is clearly positive recurrent and it is not difficult to calculate its invariant probability measure $\Pi = \left( \Pi(n) \right)_{n \geq 0}$. From the recurrence relation 
\[
\Pi(n) = \frac{1}{2 \kappa} \left( \Pi(n-1) + \Pi(n+1) \right) \quad \textnormal{for}~n \geq 2
\]
\noindent we deduce the representation
\begin{align*}
\Pi(n) = c_1 \left( \kappa - \sqrt{ \kappa^2 - 1 } \right)^{n-1} \quad \textnormal{for~all~} n \geq 1,
\end{align*}
where $c_1 \in (0,\infty)$ is a fixed constant. Using the equation
\begin{align*}
\Pi(0) &= \frac{N}{\# V} + \frac{\Pi(1)}{2 \kappa} = \frac{N}{\# V} + \frac{ c_1}{2 \kappa },
\end{align*}
as well as \medskip
\begin{align}
\Pi(0) = 1  - \sum\limits_{n \geq 1} \Pi(n) = 1 - c_1 \sum\limits_{n=0}^\infty \left( \kappa - \sqrt{\kappa^2-1} \right)^n = 1 - c_1 \tau,
\end{align}
we identify
\begin{align}
c_1 &= \left( 1 - \frac{N}{\# V} \right) \left\{ \tau +  \frac{1}{2 \kappa} \right\}^{-1} =  \frac{nm}{\# V} \left\{ \tau +  \frac{1}{2 \kappa} \right\}^{-1}.
\end{align}
Observe that the transitions of $(\Delta_n)_{n \geq 0}$ yield information on the growth of the maximal height in $\mathcal{G}$. Each transition from a state $k \neq 1$ to $0$ implies that the maximal height increases by one. The same also holds for all transitions from a state $k$ to $k+1$ and all transitions from $k \neq 2$ to $1$. On the other hand, we know that a transition from $k \geq 3$ to $k-1$ will not increase the maximal height. For the transitions from $1$ to $0$ and the transition from $2$ to $1$ we do not know for sure whether the maximal height increases. However, the conditional probability for such an event is given by $N/(N+n(m-2))$ respectively $(m-2)/(m-1)$, and the events are independent.

By applying Birkhoff's ergodic theorem to the snake chain $(\Delta_n, \Delta_{n+1})_{n \geq 0}$ and using equation (3) we obtain the following expression for $\gamma ( \mathcal{G})$. 
\begin{align*}
\gamma ( \mathcal{G}) &= \# V \Bigg( \Pi(0) + \sum\limits_{n \geq 1} \Pi(n) \frac{n}{\# V} + \sum\limits_{n \geq 2} \Pi(n) \frac{N}{\# V} + \sum\limits_{1 \leq n \neq 2} \Pi(n) \frac{n (m-2)}{ \# V}   \\
& \quad  \quad \quad \quad + \Pi(1) \frac{N+n}{\# V} \frac{N}{N+n} + \Pi(2) \frac{n(m-1)}{\# V} \frac{m-2}{m-1}  \Bigg)\\
&=\# V \left( \Pi(0) + \sum\limits_{n \geq 1} \Pi(n) \frac{N+n(m-1)}{\# V}  \right)  = \# V - n \left( 1 - \Pi(0) \right).
\end{align*}
The claim now follows by inserting (9) and (10). \end{proof}

\begin{example}~\\
\begin{center}
\begin{tabular}{|M{4em}|M{6em}|M{6em}|M{6em}|M{6em}|}
\hline
\multirow{2}{*}{$\mathcal{G}$}& \multirow{2}{*}{\begin{tikzpicture}[baseline=0]
    \fill (0,0.1) circle (2pt);
    \fill (0.5,0.1) circle (2pt);
    \fill (-0.5,0.1) circle (2pt);
    \draw (0,0.1)--(0.5,0.1);
    \draw (0,0.1)--(-0.5,0.1);
  \end{tikzpicture}}&\multirow{2}{*}{\begin{tikzpicture}[baseline=0]
    \fill (-0.25,0.25) circle (2pt);
    \fill (0.25,0.25) circle (2pt);
    \fill (0.25,-0.25) circle (2pt);
    \fill (-0.25,-0.25) circle (2pt);
    \draw (0.25,0.25)--(-0.25,0.25);
    \draw (0.25,0.25)--(0.25,-0.25);
    \draw (-0.25,-0.25)--(-0.25,0.25);
    \draw (-0.25,-0.25)--(0.25,-0.25);
  \end{tikzpicture} }& \multirow{2}{*}{\begin{tikzpicture}[baseline=0]
    \fill (-0.25,0.25) circle (2pt);
    \fill (0.25,0.25) circle (2pt);
    \fill (0.25,-0.25) circle (2pt);
    \fill (-0.25,-0.25) circle (2pt);
    \draw (0.25,0.25)--(-0.25,0.25);
    \draw (0.25,0.25)--(0.25,-0.25);
    \draw (-0.25,-0.25)--(-0.25,0.25);
    \draw (-0.25,-0.25)--(0.25,-0.25);
    \draw (-0.25,-0.25)--(0.25,0.25);
  \end{tikzpicture}}& \multirow{2}{*}{\begin{tikzpicture}[baseline=0]
    \fill (-0.25,0.25) circle (2pt);
    \fill (0.25,0.25) circle (2pt);
    \fill (0.25,-0.25) circle (2pt);
    \fill (-0.25,-0.25) circle (2pt);
    \fill (0,0) circle(2pt);
    \draw (0.25,0.25)--(0,0);
    \draw (0.25,0.25)--(-0.25,0.25);
    \draw (-0.25,-0.25)--(0.25,-0.25);
    \draw (-0.25,-0.25)--(0,0);
    \draw (-0.25,0.25)--(0,0);
    \draw (0.25,-0.25)--(0,0);
  \end{tikzpicture}}\\ &&&&\\ \hline
\multirow{2}{*}{$ \gamma ( \mathcal{G})$}& \multirow{2}{*}{$2 + \frac{1}{\sqrt{5}}$}&\multirow{2}{*}{$2 + \frac{2}{\sqrt{3}}$}&\multirow{2}{*}{$3 + \frac{1}{\sqrt{3}}$}& \multirow{2}{*}{$\frac{11}{3}$}\\ &&&& \\ \hline
\multirow{2}{*}{$ (N,n,m)$}& \multirow{2}{*}{$(1,1,2)$}&\multirow{2}{*}{$(0,1,4)$}&\multirow{2}{*}{$(2,1,2)$}& \multirow{2}{*}{$(1,2,2)$}\\ &&&& \\ \hline 
\end{tabular}~\\~\\~\\~\\

\begin{tabular}{|M{4em}|M{6em}|M{6em}|M{6em}|M{6em}|}
\hline
\multirow{4}{*}{$\mathcal{G}$}& \multirow{4}{*}{\begin{tikzpicture}[baseline=0]
	\fill (-0.5,-0.3) circle (2pt);
	\fill (-0.5,0.3) circle (2pt);	
	\fill (0,0.6) circle (2pt);	
	\fill (0,-0.6) circle (2pt);
	\fill (0.5,-0.3) circle (2pt);
	\fill (0.5,0.3) circle (2pt);
	\draw (-0.5,-0.3)--(0,-0.6)--(0.5,-0.3)--(0.5,0.3)--(0,0.6)--(-0.5,0.3)--(-0.5,-0.3);	
	\draw (-0.5,-0.3)--(0.5,-0.3)--(0,0.6)--(-0.5,-0.3);
	\draw (-0.5,0.3)--(0.5,0.3)--(0,-0.6)--(-0.5,0.3);
  \end{tikzpicture}}&\multirow{4}{*}{\begin{tikzpicture}[baseline=0]
    \fill (-0.5,0.5) circle (2pt);
    \fill (0.5,0.5) circle (2pt);
    \fill (0.5,-0.5) circle (2pt);
    \fill (-0.5,-0.5) circle (2pt);
    \fill (0,0) circle (2pt);
    \draw (0,0)--(0.5,0.5)--(-0.5,0.5)--(0,0);
    \draw (0.5,0.5)--(0.5,-0.5)--(0,0);
    \draw (0,0)--(-0.5,-0.5)--(-0.5,0.5);
    \draw (-0.5,-0.5)--(0.5,-0.5);
  \end{tikzpicture} }& \multirow{4}{*}{\begin{tikzpicture}[baseline=0]
    \fill (-0.5,0.5) circle (2pt);
    \fill (0.5,0.5) circle (2pt);
    \fill (0.5,-0.5) circle (2pt);
    \fill (-0.5,-0.5) circle (2pt);
    \fill (0.5,0) circle (2pt);
    \fill (-0.5,0) circle (2pt);
    \draw (-0.5,0)--(0.5,0.5)--(-0.5,0.5)--(0.5,0);
    \draw (0.5,0.5)--(0.5,-0.5)--(0.5,0);
    \draw (-0.5,0)--(-0.5,-0.5)--(-0.5,0.5);
    \draw (-0.5,-0.5)--(0.5,-0.5)--(-0.5,0);
    \draw (-0.5,-0.5)--(0.5,0);
    \draw (-0.5,0)--(0.5,0);
  \end{tikzpicture}}& \multirow{4}{*}{\begin{tikzpicture}[baseline=0]
    \fill (-0.5,0.5) circle (2pt);
    \fill (0.5,0.5) circle (2pt);
    \fill (0.5,-0.5) circle (2pt);
    \fill (-0.5,-0.5) circle (2pt);
    \fill (0,0.3) circle (2pt);
    \fill (0,-0.3) circle (2pt);
    \fill (0,0) circle(2pt);
    \draw (0.5,0.5)--(0,0);
    \draw (0.5,0.5)--(-0.5,0.5);
    \draw (-0.5,-0.5)--(0.5,-0.5);
    \draw (-0.5,-0.5)--(0,0);
    \draw (-0.5,0.5)--(0,0);
    \draw (0.5,-0.5)--(0,0);
    \draw (0,-0.3)--(0,0.3);
    \draw (-0.5,-0.5)--(0,-0.3)--(0.5,-0.5);
    \draw (-0.5,0.5)--(0,0.3)--(0.5,0.5);
  \end{tikzpicture}}\\ &&&&\\ &&&&\\ &&&&\\ \hline
\multirow{2}{*}{$ \gamma ( \mathcal{G})$}& \multirow{2}{*}{$3+ \frac{3}{\sqrt{2}}  $}&\multirow{2}{*}{$3 + \frac{2 \sqrt{21}}{7}  $}&\multirow{2}{*}{$4 + \frac{2}{\sqrt{5}} $}& \multirow{2}{*}{$4 + \frac{3}{\sqrt{13}}$}\\ &&&& \\ \hline 
\multirow{2}{*}{$ (N,n,m)$}& \multirow{2}{*}{$(0,1,6)$}&\multirow{2}{*}{$(1,1,4)$}&\multirow{2}{*}{$(2,2,2)$}& \multirow{2}{*}{$(1,3,2)$}\\ &&&& \\ \hline
\end{tabular}
\end{center}
\end{example}

\section{A Central Limit Theorem around $\gamma ( \mathcal{G})$}

In order to state the main result of this section let us introduce some notation. We will write $ Z_n \Longrightarrow Z$ for $n \rightarrow \infty$ to denote convergence in distribution. For $\sigma^2 \in [0,\infty)$ we denote by $N(0, \sigma^2)$ the centred normal distribution with variance $\sigma^2$. In case of $\sigma^2 = 0$ we identify the normal distribution with a Dirac measure. 

\begin{theorem}
Fix $\mathcal{G} = (V,E)$. Then there is a $\sigma^2 = \sigma^2 ( \mathcal{G}) \in [0,\infty)$, so that \[ \frac{ \max_{x \in V} H_{x,n } - n~ \frac{\gamma ( \mathcal{G} )}{\# V} }{ n^{1/2} } \Longrightarrow N(0, \sigma^2) \quad \textnormal{for}~ n \rightarrow \infty.  \]
The same central limit theorem with the same constant $\sigma^2$ also holds, if we replace $\max_{x \in V} H_{x,n}$ by $\min_{x \in V} H_{x,n}$. Moreover, 
\[ \sigma^2 ( \mathcal{G} ) = 0 \quad \textnormal{if~and~only~if} \quad \mathcal{G}~ \textnormal{is~isomorphic~to~} \mathcal{K}_{\# V}.   \]
\end{theorem}

\noindent All results of this chapter rely on the surface process $\delta = ( \delta_{n})_{n \geq 0}$ defined by
\begin{equation}
\delta_{n} :=  \left( \delta_{x,n} \right)_{x \in V}, \quad \quad \delta_{x,n} :=    H_{x,n} - \min\limits_{y \in V} H_{y,n}. 
\end{equation}
The process $(\delta_n)_{n \geq 0}$ is a time-homogeneous Markov chain, which can be seen as follows. Define an equivalence relation on the state space of $(H_n)_{n \geq 0}$ by identifying two different height vectors if and only if the height difference is the same for all vertices. Then the transition probabilities of $(H_n)_{n \geq 0}$ from one equivalence class to another do not depend on the representative of the former one. This induces a Markov chain on the set of equivalence classes, which can be identified with $(\delta_n)_{n \geq 0}$ after describing each equivalence class by its unique normalized representative.

\noindent By modifying the start height $H_0$ respectively $\delta_0$, we can ensure that the Markov chain $(\delta_n)_{n \geq 0}$ becomes irreducible with state space
\begin{equation*}
S := \left\{ (h_x)_{x \in V} \in \mathbb{N}_0^V~\vert~h_x = 0~\textnormal{for~a~} x \in V~\textnormal{and}~h_y \neq h_z~ \textnormal{for}~\{y,z \} \in E  \right\}.
\end{equation*}
The transition probabilities of $(\delta_n)_{n \geq 0}$ can be described as follows. Fix $h = (h_x)_{x \in V} \in S$ and $y \in V$ and set $m_y := \min \{ h_x~\vert~x \neq y\}$. Then we have 
\begin{align}
\mathbb{P} \left[ \delta_{n+1} = \tilde{h}~\vert~ \delta_n = h \right] = \frac{1}{\# V}, \quad \tilde{h}_x := \left\{ \begin{array}{l l}
h_x - m_y,&x \neq y\\ \\
1 + \max\limits_{z \in [y]} h_z - m_y,& x=y
\end{array} \right. .
\end{align}
Observe that for the transition from $h$ to $\tilde{h}$ there is a unique vertex $x \in V$ with $\tilde{h}_x > h_x$, and this vertex is given by $x=y$. In particular, given $h$, the state $\tilde{h}$ defined in (12) is uniquely determined by the choice of $y \in V$, and vice versa, and all non-zero transition probabilities of $(\delta_n)_{n \geq 0}$ can be described as in (12).

We will prove Theorem 2 by applying renewal arguments to $(\delta_n)_{n \geq 0}$ and using a random index central limit theorem. For fixed $h \in S$ we define the sequence of hitting times of $h$ by 
\[ \tau^h_1 := \inf \{ n \geq 0 ~\vert~ \delta_n = h \}, \quad \quad \tau_{k+1}^{h} := \inf \{ n > \tau_k^h~\vert~ \delta_n = h \}.    \]
Before starting to analyse the Markov chain $(\delta_n)_{n \geq 0}$ formally, let us briefly mention a simple but rather important observation. Let $h \in S$ and $(x_1,\ldots,x_{\# V} )$ be a non-decreasing permutation of $V$. Further assume that $h_{x_1} = \max_{y \in V} h_y$. Then, if the event $\{ \delta_n = h \}$ occurs, and in the following steps the height of the vertices $x_1,\ldots,x_{\# V}$ are increased one after the other exactly one time per vertex, this implies $\{ \delta_{n + \# V} \in S_0 \}$, where ${S}_0 \subseteq S$ denotes the set of all $h' \in S$ with $\max_{x \in V} h'_x \leq \# V$. We therefore say that $(x_1,\ldots,x_{\# V})$ resets $h$.

\begin{lemma}
Fix $\mathcal{G} = (V,E)$. Then, for all $h \in S$, there is a $n \in \mathbb{N}$ such that 
\begin{equation}
\mathbb{P} \left[ \delta_{n} = h~|~\delta_0 = \tilde{h} \right] \geq (\# V)^{-n} \quad \textnormal{for~all~} \tilde{h} \in S.
\end{equation}
In particular, each random variable $\tau_1^{h}$, $h \in S$, has an exponential moment, which is finite for any initial distribution on $S$. In particular, the Markov chain $(\delta_n)_{n \geq 0}$ is positive recurrent and has a stationary solution $\pi$.
\end{lemma}

\begin{proof}
Fix $h \in S$. Since $S_0$ is a finite subset of $S$, there is a $n_0 \in \mathbb{N}$ such that for all $h' \in S_0$ the chain $(\delta_n)_{n \geq 0}$ can go from $h'$ to $h$ in $n(h') \leq n_0$ steps. Set $n := n_0 + \# V$. Let $\tilde{h} \in S$ be given. Then fix a $x_1 \in V$ with $h_{x_1} = \max_{y \in V} h_y$ and a non-decreasing permutation $(x_1,\ldots,x_{\# V})$ of $V$. Denote by $h' \in S_0$ the unique element which arises when $\tilde{h}$ is reset according to $(x_1,\ldots,x_{\# V})$ and $m:= n_0 - n(h') \in \mathbb{N}_0$. Now assume that $\delta_0 = \tilde{h}$ and in the first $m$ step only the height of $x_1$ increases. Then we arrive at a new state, which again can be reset according to $(x_1,\ldots,x_{\# V})$, and, by construction, we therefore can go from $\tilde{h}$ to $h'$ in $m + \# V$ steps. By assumption the chain $(\delta_n)_{n \geq 0}$ may go from $h'$ to $h$ in $n(h')$ steps, and this completes the proof. \qedhere
\end{proof}

The following Lemma ensures that the Markov chain $(\delta_n)_{n \geq 0}$ contains enough information about the growth of the process $(H_n)_{n \geq 0}$.

\begin{lemma}
There are function $g_1: S \times S \rightarrow \mathbb{N}_0^V$, $g_2: S \times S \rightarrow \{0,1\}$ with
\begin{align*}
g_1(\delta_n,\delta_{n+1}) &= H_{n+1} - H_n \quad \textnormal{and}\\
g_2 ( \delta_n, \delta_{n+1}) &= \max_{x \in V} H_{x,n+1} - \max_{y \in V} H_{y,n}~~\textnormal{almost~surely}.
\end{align*}
\end{lemma}

\begin{proof}
The main step is to define $g_1(h,\tilde{h})$, since given $g_1$ we can construct $g_2$ for example by the formula
\[
g_2 (h,\tilde{h}):= \sum\limits_{x \in V} \left( g_1 (h,\tilde{h}) \right)_x~\textbf{1}_{\max\limits_{x' \in [x]} h_{x'} = \max\limits_{z \in V} h_z}.
\]
For the definition of $g_1$ recall the description of the transition probabilities of $(\delta_n)_{n \geq 0}$ given above in formula (12). Each transition of $(\delta_n)_{n \geq 0}$ corresponds to an increase of the value of one uniquely vertex, and clearly this also holds for all transitions of our deposition process $(H_n)_{n \geq 0}$. By recalling our definition of $(\delta_n)_{n \geq 0}$ in (11) it is clear that these two vertices are always the same. Now let $h$, $ \tilde{h} \in S$ and $y \in V$ be given as in (12). Then we can define $g_1$ by
\begin{equation*}
\left( g_1(h,\tilde{h}) \right)_x:= \left\{ \begin{array}{l l}
1 + \max\limits_{z \in [x]} h_z - h_x,&x = y\\ \\
0,&x \neq y.
\end{array} \right. \qedhere
\end{equation*}
\end{proof}

\begin{example}
Let $\mathcal{G} = \mathcal{S}_3$. Then we can simplify the structure of the Markov chain $(\delta_n)_{n \geq 0}$ by stepwise performing the following manipulations.\\
\begin{tabular}{l l}
(i)&Always identify $(h_1,h_2,h_3)$ and $(h_3,h_2,h_1)$ with each other.\\ \\
(ii)&Always identify $(h_1,0,h_3)$ with $(h'_1,0,h'_3)$ if $\vert h_1-h_3 \vert = \vert h'_1 - h'_3 \vert$.\\ \\
(iii)&Always identify $(0,h_2,h_3)$ and $(0,h'_2,h'_3)$ if $h_2 - h_3 = h'_2 - h'_3$.\\ \\
(iv)&Replace all states of the form $(0,h_2,h_3)$ with $h_2 > h_3$ by a new single\\
&state $z$. If the chain is in $z$, it will stay in $z$ with probability $1/3$.\\ \\
(v)&Identify $(h_1,0,h_3)$ with $(0,h'_2,h'_3)$ if $h'_3 - h'_2 = |h_1-h_3|$.\\~\\
\end{tabular}

\noindent After these simplifications we arrive at a Markov chain, whose state space and transition probabilities are illustrated in the following picture.

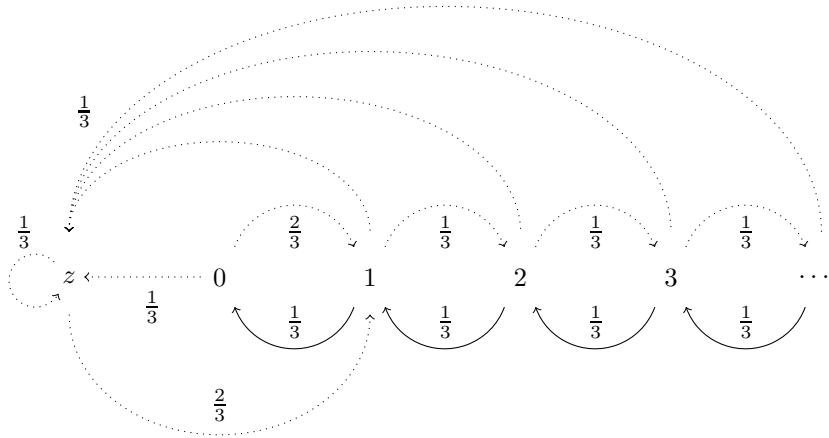
\begin{figure}[!h]
\begin{center}
\begin{tikzpicture}[scale=2]
\node (0) at (0,0) {$z$};
\node (1) at (1,0) {$0$};
\node (2) at (2,0) {$1$};
\node (3) at (3,0) {$2$};
\node (4) at (4,0) {$3$};
\node (...) at (4.95,0) {$\ldots$};
\draw [dotted,->] (1.1,0.2) arc (160:20:12pt);
\draw [dotted,->] (2.1,0.2) arc (160:20:12pt);
\draw [dotted,->] (3.1,0.2) arc (160:20:12pt);
\draw [dotted,->] (4.1,0.2) arc (160:20:12pt);

\draw [<-] (1.1,-0.2) arc (200:340:12pt);
\draw [<-] (2.1,-0.2) arc (200:340:12pt);
\draw [<-] (3.1,-0.2) arc (200:340:12pt);
\draw [<-] (4.1,-0.2) arc (200:340:12pt);
\draw [dotted,->] (-0.1,0.1) arc (45:330:5pt);

\draw[<-,dotted] (0,0.3) arc(180:0:1cm and 0.6cm);
\draw[<-,dotted] (0,0.3) arc(180:0:1.5cm and 0.9cm);
\draw[<-,dotted] (0,0.3) arc(180:0:2cm and 1.2cm);
\draw[<-,dotted] (0,0.3) arc(180:0:2.5cm and 1.5cm);

\node at (-0.3,0.3) {$\frac{1}{3}$};
\node at (1.5,0.3) {$\frac{2}{3}$};
\node at (2.5,0.3) {$\frac{1}{3}$};
\node at (3.5,0.3) {$\frac{1}{3}$};
\node at (4.5,0.3) {$\frac{1}{3}$};
\node at (0.55,-0.2) {$\frac{1}{3}$};
\node at (1.5,-0.3) {$\frac{1}{3}$};
\node at (2.5,-0.3) {$\frac{1}{3}$};
\node at (3.5,-0.3) {$\frac{1}{3}$};
\node at (4.5,-0.3) {$\frac{1}{3}$};
\node at (0.1,1.1) {$\frac{1}{3}$};

\node at (1.0,-0.85) {$\frac{2}{3}$};

\draw [dotted,->] (0,-0.25) arc (180:360:1cm and 0.8cm);

\draw [dotted, <-] (0.1,0) -- (0.9,0);
\end{tikzpicture}
\end{center}
\caption{Dotted arrows indicate that the maximal height grows by $1$ for each transition. Transitions along solid lines do not increase the maximal height.}
\end{figure}
\end{example}

\newpage

Our last ingredient for the proof of Theorem 2 is a rather simple inequality. For its proof we use the counterpart of the process $(\delta_n)_{n \geq 0}$ in our continuous time deposition process. We define
\[ \tilde{\delta}_t := (\tilde{\delta}_{x,t})_{x \in V}, \quad \quad \tilde{\delta}_{x,t} := \tilde{H}_{x,t} - \min_{y \in V} \tilde{H}_{y,t}.   \]
Roughly speaking, all previously mentioned arguments and results for $(\delta_n)_{n \geq 0}$ also hold for $(\tilde{\delta}_t)_{t \geq 0}$ with only minor changes.

\begin{proposition}
Let $\mathcal{G} = (V,E )$, $\mathcal{G}' = (V',E')$ be given graphs and assume that $\mathcal{G}$ is a subgraph of $\mathcal{G}'$. Then $\gamma ( \mathcal{G} ) \leq \gamma ( \mathcal{G}')$ and
\[ \gamma ( \mathcal{G}) = \gamma ( \mathcal{G}') \quad \textnormal{if~and~only~if} \quad \mathcal{G} = \mathcal{G}'.   \]
\end{proposition}

\begin{proof}
We couple our ballistic deposition processes on $\mathcal{G}$ and $\mathcal{G}'$ in continuous time by assuming that they share the same underlying Poisson processes $(\xi_x)_{x \in V}$. This directly gives $\gamma ( \mathcal{G}) \leq \gamma ( \mathcal{G}')$.

\noindent Now assume $\mathcal{G} \neq \mathcal{G}'$ and let us prove $\gamma ( \mathcal{G} ) < \gamma ( \mathcal{G}')$. For this purpose, note that by induction over $\# V$ it suffices to consider the following two cases.

\noindent \begin{tabular}{l l}
(i)&$V' = V$ and $E' = E \cup \{ \{ x,y \} \}$ for a suitable choice of $x$, $y \in V$.\\ \\
(ii)&$V' = V \cup \{ x' \}$ for a $x' \notin V$ and $E' = E \cup \{ \{ x, x' \} \}$ for a $x \in V$.\\
\end{tabular}\\ \\

\noindent It turns out that both cases can be treated roughly in the same way, and we therefore start and mainly concentrate on the case (ii).

To obtain the claim we construct a new growth model on $\mathcal{G}'$, which evolves asymptotically faster than our deposition model on $\mathcal{G}$ and at most as fast as the deposition process on $\mathcal{G}'$. For this purpose let $\xi_{x'}$ denote the Poisson process related to the vertex $x'$ in the later one. Furthermore, let $(\tilde{\delta}_t)_{t \geq 0}$ denote the time continuous surface process related to the ballistic deposition on the graph $\mathcal{G}$. Further fix a $h \in S$ with $h_x = \max_{y \in V} h_y$ and a non-decreasing permutation $(x_1,\ldots,x_{\# V})$ of $V$ with $x = x_1$.

Our new growth process on $\mathcal{G}'$ arises by modifying our ballistic deposition rule (1). We will take the possible growth events of the vertex $x'$ as well as the influence of $x'$ on its neighbour $x$ only into account, if the current height fluctuations behave in a specific way. More precisely the influence of $x'$ at a point in time is only taken into account if both the Markov chain $(\tilde{\delta}_t)_{t \geq 0}$ is in the state $h$, and then, in the following, the first Poisson process, who jumps, is $\xi_{x'}$, followed by a jump of $\xi_{x_1}$, $\xi_{x_2}$, and so on until $\xi_{x_{\# V}}$ has jumped. After such an event has occurred, we again neglect the possible growth of $x'$ or its influence on the growth of $x=x_1$, until at a later point in time again both $(\tilde{\delta}_t)_{t \geq 0}$ is in $h$ and in the following the Poisson processes behave accordingly.

By definition, it is clear that the height of our new growth process at all times is smaller than in our original ballistic deposition process on $\mathcal{G}'$. This is just a consequence of the fact, that the influence of the vertex $x'$ is always taken into account in our original process.

On the other hand, the maximal height in our new processes always exceeds the maximal height of our ballistic deposition on $\mathcal{G}$, since in this process the vertex $x'$ is always neglected. However, by construction, we know that the maximal height of our new process at a time $t \in (0,\infty)$ is always at least as big as the maximal height in our ballistic deposition on $\mathcal{G}$ plus the number of visits of $(\tilde{\delta}_t)_{t \geq 0}$ in $h$ up to time $t$, which have been followed by the above mentioned behaviour of the underlying Poisson processes. Since $(\tilde{\delta}_t)_{t \geq 0}$ is a positive recurrent Markov chain and irreducible on $S$, Birkhoff's ergodic theorem yields that this second contribution strictly increases the asymptotic growth. This implies $\gamma ( \mathcal{G}) < \gamma ( \mathcal{G}')$ and therefore verifies our claim in case (ii). 

The case (i) can be treated roughly in the same way as (ii). Instead of taking the growth of the vertex $x'$ and its influence on $x$ into account only sometimes, one now has to handle the influence of the edge $\{ x,y\}$ in a similar way. \qedhere
\end{proof}

\begin{proof}[Proof of Theorem 2]
Lemma 2 gives us the representation
\begin{equation}
R_n := \max\limits_{x \in V} H_{x,n} - n \frac{\gamma( \mathcal{G})}{\# V} = \max_{x \in V} H_{x,0} + \sum\limits_{k=1}^n f( \delta_{k-1},\delta_k ), \quad n \geq 0,
\end{equation}
where $f(h,h'):=g_2(h,h') - \frac{\gamma ( \mathcal{G})}{\# V}$. Fix $h \in S$ and assume $H_0 := \delta_0 :=h$. Consider the sequence $(W_n)_{n \geq 1}$ defined by
\[ W_n := f \left( \delta_{\tau^h_{n}}, \delta_{\tau_h^{n}+1} \right) + \cdots + f \left( \delta_{\tau^h_{n+1}-1}, \delta_{\tau^h_{n+1}} \right).  \]
The random variables $(W_n)_{n \geq 1}$ are i.i.d. by construction. Besides, since clearly $- 1 \leq f \leq 1$, Lemma 1 yields

\begin{equation*}
0 \leq \tilde{\sigma}^2 :=  \mathbb{E} \left[ W_1^2 \right] \leq \mathbb{E} \left[ ( \tau_2^h - \tau_1^h)^2 \right] < \infty . \medskip
\end{equation*}

\noindent Set $K_n := \sup \{ k \in \mathbb{N}~|~ \tau^h_k \leq n \}$. Let us prove the following statements. 

\begin{tabular}{l l}
A)&$\displaystyle \frac{1}{n^{1/2}}~R_{\tau^h_{K_n}} \Longrightarrow N(0,\sigma^2)$ for $n \rightarrow \infty$, where $\sigma^2 := \pi(h)~\tilde{\sigma}^2$.\\ \\
B)&$\displaystyle \frac{1}{n^{1/2}}  \left\vert R_n - R_{\tau^h_{K_n}} \right\vert \Longrightarrow 0 $ for $n \rightarrow \infty$.\\ \\
\end{tabular}

Once we have established A) and B), Slutsky's theorem immediately gives
\[  \frac{1}{n^{1/2}} R_n = \frac{1}{n^{1/2}} R_{\tau^h_{K_n}} + \frac{1}{n^{1/2}} \left( R_n - R_{\tau^h_{K_n}} \right) \Longrightarrow N(0,\sigma^2)  \quad \textnormal{for}~ n \rightarrow \infty, \medskip \]
and hence verifies the first claim of Theorem 2.

\noindent In order to prove A), note that by Kac's theorem we know $\tau_n^h \sim \pi(h)^{-1} \cdot n$ for $n \rightarrow \infty$ almost surely and $K_n \sim \pi(h) \cdot n$ for $n \rightarrow \infty$. Moreover, we have 
\[
\frac{1}{n^{1/2}} R_{\tau^h_{K_n}} = \frac{1}{n^{1/2}} \sum\limits_{k=1}^{K_n} W_k.
\]
Taking these observations into account, the statement A) directly follows from Anscombe's theorem, see e.g. Theorem 2.3 in \cite{Gut12}. 

\noindent In order to prove B) consider the estimate
\[ \frac{1}{n^{1/2}} \left| R_n - R_{\tau^h_{K_n}} \right| \leq \frac{1}{n^{1/2}} \sup \left\{ \left\vert R_{k} - R_{\tau^h_{K_n}} \right\vert;~k=\tau^h_{K_{n}}, \tau^h_{K_n} + 1, \ldots,\tau^h_{K_{n}+1} \right\}. \medskip \]
Now we apply the inequality 
\[ \left| R_k - R_l \right| \leq \max \left\{ \frac{\gamma(\mathcal{G})}{\# V}, 1 - \frac{\gamma(\mathcal{G})}{\# V} \right\} \left| k-l \right| \leq |k-l|, \quad k,l \geq 1, \medskip \]
and obtain
\[ \frac{1}{n^{1/2}} \left| R_n - R_{\tau^h_{K_n}} \right| \leq  \frac{1}{n^{1/2}} \left( \tau_{K_n+1}^h - \tau_{K_n}^h \right).    \]
The Markov property of $(\delta_n)_{n \geq 0}$ implies that the sequence $\big( \tau_{K_n + 1}^h - \tau_{K_n}^h \big)_{n \geq 1}$ is i.i.d. and therefore we have verified the claim (B). 

Until now we have verified the central limit theorem for the maximal height for a deterministically chosen initial state of $H_0$ respectively $\delta_0$. The following argument shows that changing the initial distribution will not affect the correctness of the claim. Consider two ballistic deposition processes on $\mathcal{G}$ with different deterministic initial values, and couple them by assuming that with each step the height of the same vertex is increased. Then, by our deposition rule (1), the maximal height difference cannot increase over time. So, if a central limit theorem holds for one process, it also holds for the other one. Clearly, these arguments can be extended to also allow an arbitrary initial distribution on $S$.

\noindent Let us now continue with the second claim of Theorem 2. We start will start by proving it under the assumption $\delta_0 \sim \pi$. Note that
\begin{equation}
\frac{\min\limits_{x \in V} H_{x,n} - n \frac{\gamma(\mathcal{G})}{\# V}}{ n^{1/2} } = \frac{\max\limits_{x \in V} H_{x,n} - n \frac{\gamma(\mathcal{G})}{\# V}}{ n^{1/2} } - \frac{\max\limits_{x \in V} \delta_{x,n} }{n^{1/2}}.
\end{equation}~

Since $\delta_0 \sim \pi$ we know that the distribution of $\max_{x \in V} \delta_{x,n}$ does not depend on $n$. Therefore the claim follows by applying Slutsky's theorem to (15) and using the central limit theorem for the maximal height.

So far we have verified the central limit theorem for the minimal height under the assumption $\delta_0 \sim \pi$. Since $\pi(h) > 0$ for all $h \in S$ we conclude as above that the central limit theorem holds for arbitrary initial distributions.

Now let us finish with the last claim of Theorem 2. If $\mathcal{G}$ is isomorphic to a complete graph, then clearly $\sigma^2 = 0$. Hence we can assume that $\mathcal{G}$ is not isomorphic to $\mathcal{K}_{\# V}$. Recall that we established the representation $\sigma^2 = \tilde{\sigma}^2 \pi(h)$ in A). Therefore, we only need to guarantee that $\tilde{\sigma}^2 = \mathbb{E} [W_1^2] > 0$. 

Fix a $x \in V$ with $h_x = \max_{y \in V} h_y$ and a non-decreasing permutation $(x_1,\ldots,x_{\# V})$ of $V$ with $x_1 = x$. Denote by $\tilde{h} \in S$ the unique state, at which the chain $(\delta_n)_{n \geq 0}$ arrives after starting in $h$ and being reset according to $(x_1,\ldots,x_{\# V})$. Since $(\delta_n)_{n \geq 0}$ is irreducible, there is a finite path along which $(\delta_n)_{n \geq 0}$ may go from $\tilde{h}$ and $h$. Let $h_1,h_2,\ldots,h_N$ denote the path with $h_1 = h_N = h$ which arises by concatenation. Let $M \in \mathbb{N}$ be the number of returns of $(\delta_n)_{n \geq 0}$ to $h$ along this path. Then we know
\[ \mathbb{P} \left[ W_1 + \ldots + W_M = c \right] > 0, \quad \textnormal{where}~ c := \sum\limits_{j=1}^{N-1} f(h_j,h_{j+1}). \] 
If $c \neq 0$, then clearly $\tilde{\sigma}^2 > 0$ and the claim holds true. Therefore, we can assume that $c=0$. Then, we continue with construction of another path $h'_1,\ldots,h'_{N+1}$, along which $(\delta_n)_{n \geq 0}$ can go from $h'_1 := h$ to $h'_{N+1}:=h$. For this purpose let $h'_1:=h_1$ and note that the permutation $(x_1,\ldots,x_{\# V})$ still resets $h'_1$. We define $h'_2,\ldots,h'_{\# V+2}$ by using the resetting event and note that $h'_{\# V + 2} = \tilde{h}$. Then we consider the same path from $\tilde{h}$ to $h$ as before and therefore we can set $h'_{k} := h_{k-1}$ for all $1 \leq k \leq N+1$. Let $M' \in \mathbb{N}$ be the number of returns of $(\delta_n)_{n \geq 0}$ to $h$ along $h'_1,\ldots,h'_{N+1}$. Then, by construction,
\[
\mathbb{P} \left[ W_1 + \cdots + W_{M'} = c' \right]  > 0, \quad \textnormal{where}~~ c' := \sum\limits_{j=1}^{n+1} f(h'_j, h'_{j+1}).
\]
By construction of our route $h'_1, \ldots, h'_{n}$ and our assumption $c=0$ we have
\[
c' = \sum\limits_{j=1}^{n+1} f(h'_j, h'_{j+1}) =  1 - \frac{\gamma ( \mathcal{G})}{\# V} + c  = 1 - \frac{\gamma ( \mathcal{G})}{\# V}.
\]
By Proposition 1 we know $\gamma ( \mathcal{G}) < \gamma ( \mathcal{K}_{\# V} ) = \# V$ and therefore $c' > 0$. Clearly, this implies $\tilde{\sigma}^2 = \mathbb{E} [W_1^2] > 0$ and therefore $\sigma^2 > 0$. \qedhere
\end{proof}

\begin{remark}
There is also a slightly different approach towards Theorem 2. The process $(R_n)_{n \geq 0}$ defined by (14) is stationary under the assumption $\delta_0 \sim \pi$. To obtain the first claim of Theorem 2, one therefore only needs to ensure adequate moment and mixing conditions, compare e.g. Theorem 27.5 in \cite{Bil79}.

Let $Y_n := R_{n+1} - R_n$. Then $\mathbb{E}_\pi [Y_n] = 0$ by Birkhoff's ergodic theorem and $-1 \leq Y_n \leq 1$ almost surely. Moreover, equation (13) verifies a so-called Doeblin condition of the Markov chain $(\delta_n)_{n \geq 0}$, which implies a nice form of geometric ergodicity, see Chapter 2 in \cite{Loc13}. This ergodicity in return implies exponentially fast mixing, see Chapter 1 and Theorem 3.7 in \cite{Bra05}. Alternatively, one can also obtain this result by considering the coefficient of ergodicity of $(\delta_n)_{n \geq 0}$. For more information we refer to Chapter 2 in \cite{Loc13} and Chapter 3 in \cite{FG20}.

As a consequence of the central limit theorem for stationary processes we obtain the representation
\begin{align}
\sigma^2 = \textnormal{Var}_\pi [Y_1^2] + 2 \sum\limits_{k=2}^\infty \textnormal{Cov}_\pi[Y_1 Y_k] \in [0,\infty).
\end{align}
However, it seems rather difficult to prove that $\sigma^2 ( \mathcal{G} ) > 0$ if $\mathcal{G}$ is not isomorphic to a complete graph by only using (16).

In fact the above mentioned moment and mixing conditions do not only imply the classical central limit theorem, but also the functional version of the central limit theorem, compare Corollary 1 in \cite{Her84}, as well as a law of the iterated logarithm, see Theorem 2 and the further comments in \cite{Rio95}.
\end{remark}

\section{A general upper bound for $\gamma ( \mathcal{G})$}

The following result is based on the arguments used by Atar, Athreya and Kang in \cite{AAK01} to derive the upper bound in (4).

\begin{theorem}
Let $\mathcal{G}$ be a given graph. Let $\rho$ be the spectral radius of $A(\mathcal{G}) + \mathbf{1}$, where $\mathbf{1}$ denotes the identity matrix with respect to the index set $V$. Then,
\[ \gamma ( \mathcal{G} ) \leq e \cdot \rho.    \]
\end{theorem}

\begin{proof}
Fix $m \in \mathbb{N}$ and let
\[ T_m := \inf  \left\{ t >0 ~|~ \max\limits_{x \in V} \tilde{H}_{x,t} = m \right\}.  \]
We modify our time continuous deposition process in the following way. At time $T_m$ the height in each vertex is set equal to $m$. Then the process evolves as usual, until the maximal height again hits a multiple of $m$. At this particular time the height of each vertex of the graph is increased, until it is again equal to the maximal height. By continuing this procedure, we arrive at a model, which grows at least as fast as our original process, and by the law of large numbers
\[ \gamma ( \mathcal{G} )  \leq \frac{m}{ \mathbb{E} [T_m]  }.   \]
Applying Markov's inequality, we find for all $a \in (0,\infty)$ 
\[ \mathbb{E} [ T_m ] \geq am \left( 1 - \mathbb{P} [T_m \leq am] \right)   \]
 and therefore
\begin{equation}
\gamma( \mathcal{G}) \leq \frac{1}{a \left( 1 - \mathbb{P} [T_m \leq am] \right) }. 
\end{equation}
Note that $T_m \leq am$ if and only if there are $x_1,\ldots,x_m \in V$ and $0 < t_1 < \ldots < t_m \leq am$ such that $x_{i+1} \in [x_i]$ for all $i=0,\ldots,m-1$ and in each time interval $(t_i,t_{i+1})$ the height of $x_i$ increases strictly. The number of tuples $(x_1,\ldots,x_m)$ satisfying $x_{i+1} \in [x_i]$ for all $i=0,\ldots,m-1$ is given by $\lVert \left( A(\mathcal{G}) + \mathbf{1} \right)^m \rVert$, where the norm $\lVert \cdot \rVert$ is defined as the sum of the absolute value of all entries. Observe that $A(\mathcal{G}) + \mathbf{1}$ is a nonnegative irreducible matrix. Therefore, by the Perron-Frobenius theorem
\[
\rho = \lim\limits_{n \rightarrow \infty} \sqrt[n]{\left\lVert \left( A(\mathcal{G}) +  \mathbf{1} \right)^n \right\rVert}.
\]
We conclude that for all $\varepsilon > 0$ there exists a $m_0 \in \mathbb{N}$, such that for $m \geq m_0$
\[
\mathbb{P} [T_m \leq am] \leq ( \rho + \varepsilon)^m \mathbb{P} \left[ S_m \leq am \right],
\]
where $S_m = \sum_{k=1}^m W_k$ and $(W_n)_{n \geq 1}$ is a sequence of i.i.d. exponentially distributed random variables of unit mean. Using Markov's inequality, we have for all $\lambda \in (0,\infty)$ the estimate
\begin{align*}
\mathbb{P} [S_m \leq am ] &= \mathbb{P}  \left[ \exp ( -\lambda S_m) > \exp(- a \lambda m) \right] \leq \exp( a \lambda m)  \mathbb{E} \left[ \exp(-\lambda W_1)\right]^m\\
&= \exp (  a \lambda m) \left( \frac{1}{1 + \lambda} \right)^m = \exp \left( \left( a \lambda  - \log(1 + \lambda) \right) m  \right).
\end{align*}
Minimizing over $\lambda$ we find the optimal bound for $\lambda = (1-a)/a$ and
\[ \mathbb{P}  [T_m \leq am] \leq \# V \exp \left( m \left(  1 - a + \log(a) + \log( \rho + \varepsilon) \right)  \right). \]
Choose $a:= \frac{1}{e(\rho + \varepsilon)} \in (0,\infty)$. Then $\log(a) = - \log( \rho + \varepsilon) -1$ and therefore
\[
\mathbb{P}  [T_m \leq am] \leq \exp \left( - a m \right) \rightarrow 0 \quad \textnormal{for~} m \rightarrow \infty.
\]
Applying this to equation (17) gives
\[ \gamma ( \mathcal{G}) \leq \frac{1}{a} = e \left( \rho + \varepsilon \right).  \]
The claim now follows by letting $\varepsilon \rightarrow 0$. \qedhere \end{proof}

\begin{remark}
Note that always $\rho \leq \Delta \mathcal{G} + 1$, and equality holds if and only if $\mathcal{G}$ is a regular graph. By considering the case of a complete graph, we immediately see that the upper bound in Theorem 3 is optimal up to a constant. One might ask, whether there exists a sequence of regular graphs $(\mathcal{G}_n)_{n \geq 1}$ satisfying 
\[ \gamma ( \mathcal{G}_n) \sim e~\Delta \mathcal{G}_n \quad \textnormal{as~} n \rightarrow \infty \quad \textnormal{?}
  \]
\end{remark}

We want to give a minor result, which is again proven by studying the continuous time version of our ballistic deposition process.

\begin{proposition}
Let $(\mathcal{G}_n)_{n \geq 0}$ be a sequence of regular graphs with $\Delta \mathcal{G}_n \rightarrow \infty$ for $n \rightarrow \infty$ and \textnormal{girth}$(\mathcal{G}_n) \geq 5$ for all $n \in \mathbb{N}$. Then 
\begin{align*}
\liminf\limits_{n \rightarrow \infty} \frac{\gamma ( \mathcal{G}_n)}{\Delta \mathcal{G}_n} \geq \frac{2e-1}{(e-1)^2} \approx 1.506.
\end{align*}
\end{proposition}

\begin{proof}
Fix $M \in \mathbb{N}$ and $n_0 \in \mathbb{N}$, such that $m:= \Delta \mathcal{G}_n+1 > M$ for all $n \geq n_0$.

We construct a random growth model, which evolves slower than our original one. For simplicity, assume that the first three vertices $x_1,x_2,x_3$, which grow, form a path $(x_1,x_2,x_3)$ in $\mathcal{G}$. Now only take into account the neighbours of $x_2$ and $x_3$ and, for the time being, neglect the possible growth of any other vertex. We also neglect the growth of $x_1$, $x_2$ and $x_3$.

If a neighbour $x_4$ of $x_3$ grows, then we forget about the height of $x_1$ and consider the path $(x_2,x_3,x_4)$ instead of $(x_1,x_2,x_3)$ and keep on with our procedure. Note that by this transition the maximal height grows by one unit.
 
If a neighbour $x'_3$ of $x_2$ grows, we memorise its height. Then we will neglect further growth of $x'_3$, but in the future we will take into account its neighbours, which differ from $x_2$. If a neighbour $x'_4 \neq x_2$ of $x'_3$ grows, we will only memorise the height of $x_2$, $x'_3$ and $x'_4$ and forget the height of all other vertices. Note that by this rule, we again arrive at our initial situation.

We will memorise the height of up to $M$ neighbours of $x_2$. If this number is exhausted, we will not take into account the potential growth of a neighbour of $x_2$ anymore. By counting the number of neighbours of $x_2$ in our continuous time setting, we therefore obtain Markov process, whose transition rates are given in the following picture.\\

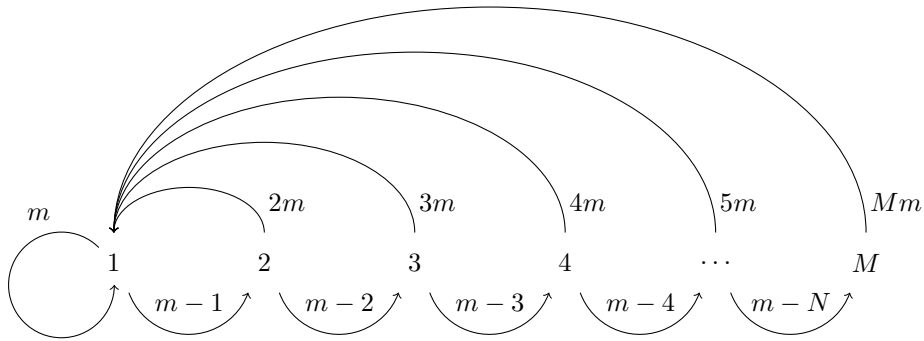
\noindent \begin{tikzpicture}[scale=2]
\node (1) at (0,0) {$1$};
\node (2) at (1,0) {$2$};
\node (3) at (2,0) {$3$};
\node (4) at (3,0) {$4$};
\node (...) at (4,0) {$\ldots$};
\node (5) at (5,0) {$M$};
\draw [->] (0.1,-0.2) arc (200:340:12pt);
\draw [->] (1.1,-0.2) arc (200:340:12pt);
\draw [->] (2.1,-0.2) arc (200:340:12pt);
\draw [->] (3.1,-0.2) arc (200:340:12pt);
\draw [->] (4.1,-0.2) arc (200:340:12pt);
\draw [->] (-0.1,0.1) arc (45:360:10pt);

\draw[<-] (0,0.2) arc(180:0:0.5cm and 0.3cm);
\draw[<-] (0,0.2) arc(180:0:1cm and 0.6cm);
\draw[<-] (0,0.2) arc(180:0:1.5cm and 0.9cm);
\draw[<-] (0,0.2) arc(180:0:2cm and 1.2cm);
\draw[<-] (0,0.2) arc(180:0:2.5cm and 1.5cm);

\node at (-0.5,0.32) {$m$};
\node at (1.15,0.4) {$2m$};
\node at (2.15,0.4) {$3m$};
\node at (3.15,0.4) {$4m$};
\node at (4.15,0.4) {$5m$};
\node at (5.2,0.4) {$Mm$};
\node at (0.5,-0.27) {$m-1$};
\node at (1.5,-0.27) {$m-2$};
\node at (2.5,-0.27) {$m-3$};
\node at (3.5,-0.27) {$m-4$};
\node at (4.5,-0.27) {$m-N$};
\end{tikzpicture}~\\~\\

Let $p_{m,M} (k)$, $k=1,\ldots,m$, denote the invariant probability distribution. Then, by including the time-scaling induced by the transition rates, we know
\begin{align*}
\gamma ( \mathcal{G}_n) \geq & \left( \sum\limits_{k=1}^{M-1}  p_{m,M} (k) \left\{  (k+1)m-k \right\} + p_{m,M} (M) Mm \right)\\
& \cdot \left( \sum\limits_{k=1}^{M-1} p_{m,M} (k) \frac{km}{km+(m-k)} + p_{m,M} (M) \right).
\end{align*}
For $n \rightarrow \infty$ we know $m \rightarrow \infty$ and therefore $p_{m,M} (k) \rightarrow p_M (k)$, where $p_M (k)$, $k=1,\ldots,M$, is the invariant probability of time discrete Markov chain, which is depicted in the following picture.\\~\\

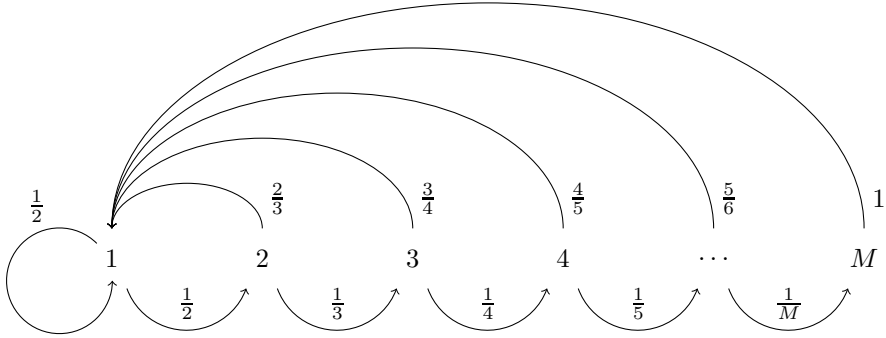
\noindent \begin{tikzpicture}[scale=2]
\node (1) at (0,0) {$1$};
\node (2) at (1,0) {$2$};
\node (3) at (2,0) {$3$};
\node (4) at (3,0) {$4$};
\node (...) at (4,0) {$\ldots$};
\node (5) at (5,0) {$M$};
\draw [->] (0.1,-0.2) arc (200:340:12pt);
\draw [->] (1.1,-0.2) arc (200:340:12pt);
\draw [->] (2.1,-0.2) arc (200:340:12pt);
\draw [->] (3.1,-0.2) arc (200:340:12pt);
\draw [->] (4.1,-0.2) arc (200:340:12pt);
\draw [->] (-0.1,0.1) arc (45:360:10pt);

\draw[<-] (0,0.2) arc(180:0:0.5cm and 0.3cm);
\draw[<-] (0,0.2) arc(180:0:1cm and 0.6cm);
\draw[<-] (0,0.2) arc(180:0:1.5cm and 0.9cm);
\draw[<-] (0,0.2) arc(180:0:2cm and 1.2cm);
\draw[<-] (0,0.2) arc(180:0:2.5cm and 1.5cm);

\node at (-0.5,0.35) {$\frac{1}{2}$};
\node at (1.1,0.4) {$\frac{2}{3}$};
\node at (2.1,0.4) {$\frac{3}{4}$};
\node at (3.1,0.4) {$\frac{4}{5}$};
\node at (4.1,0.4) {$\frac{5}{6}$};
\node at (5.1,0.4) {$1$};
\node at (0.5,-0.3) {$\frac{1}{2}$};
\node at (1.5,-0.3) {$\frac{1}{3}$};
\node at (2.5,-0.3) {$\frac{1}{4}$};
\node at (3.5,-0.3) {$\frac{1}{5}$};
\node at (4.5,-0.3) {$\frac{1}{M}$};
\end{tikzpicture}~\\

\noindent A simple calculation gives 
\[
p_M(k) = \frac{1}{k! \left( 1+1/2+1/6 + \ldots + 1/M! \right)}, \quad k=1,\ldots,M.
\]
For $m \rightarrow \infty$ we have
\begin{align*}
\frac{1}{m+1} \sum\limits_{k=1}^{M-1} p_{m,M} (k) \left\{ (k+1) m - k \right\} &\rightarrow \sum\limits_{k=1}^{M-1} p_M(k)~(k+1),\\
\sum\limits_{k=1}^{M-1} p_{m,M} (k) \frac{km}{km+(m-k)} &\rightarrow \sum\limits_{k=1}^{M-1} p_M(k) \frac{k}{k+1}.
\end{align*}
Therefore, we conclude for fixed $M$
\begin{align*}
\liminf\limits_{n \rightarrow \infty} \frac{\gamma ( \mathcal{G}_n)}{ \Delta \mathcal{G}_n + 1} \geq \left( \sum\limits_{k=1}^{M-1} p_M (k)~(k+1) \right) \left( \sum\limits_{k=1}^{M-1} p_M(k) \frac{k}{k+1} \right).
\end{align*}
Finally, we let $M \rightarrow \infty$ and note that
\begin{align*}
\sum\limits_{k=1}^{M-1} p_M(k)~(k+1) &\rightarrow \frac{1}{e-1} \sum\limits_{k=1}^\infty \frac{k+1}{k!} = \frac{2e-1}{e-1},\\
\sum\limits_{k=1}^{M-1} p_M (k) \frac{k}{k+1} &\rightarrow \frac{1}{e-1} \sum\limits_{k=1}^\infty \frac{k}{(k+1)!} = \frac{1}{e-1}.
\end{align*}
\end{proof}

\noindent Let us give a simple conclusion of both Proposition 2 and Theorem 3. 

\begin{corollary}
Let $(\mathcal{G}_n)_{n \in \mathbb{N}}$ be a sequence of graphs. Then
\[ \lim\limits_{n \rightarrow \infty} \gamma ( \mathcal{G}_n) = \infty \quad \quad \textnormal{if~and~only~if} \quad \quad \lim\limits_{n \rightarrow \infty} \Delta \mathcal{G}_n = \infty. \]
\end{corollary}

\section{A brief look at a different growth model}

For a better understanding of our growth model, it is natural to also study other ballistic deposition processes, which arise by modifying the recursion (1). The so-called nearest-neighbour ballistic deposition model is specified by 
\[ \tilde{h}_x := \max\limits_{y \in [x]} ~\{ h_y + \delta_{xy} \},   \]
where $\delta_{xy}$ is the Kronecker symbol. For a graph $\mathcal{G}$ we can define its asymptotic growth parameter $\tilde{\gamma}(\mathcal{G})$ in this new model in the same way as in the Introduction. A simple coupling argument gives $\tilde{\gamma} ( \mathcal{G}) \leq \gamma ( \mathcal{G})$ for arbitrary $\mathcal{G}$, and
\[
\gamma ( \mathcal{S}_n) \leq \tilde{\gamma}( \mathcal{S}_n) + 2 \quad \textnormal{for~all~} n \in \mathbb{N}.
\]
In particular, we see that Proposition 2, Theorem 3 and Corollary 1 also hold if we replace $\gamma ( \mathcal{G})$ by $\tilde{\gamma} ( \mathcal{G})$. However, direct calculations can reveal some differences. Consider the case $\mathcal{G}= \mathcal{K}_n$ with $n \geq 1$ fixed. Then, by counting the number of vertices of maximal height, we arrive at a Markov chain, whose transition probabilities are described in the following picture.\\~\\

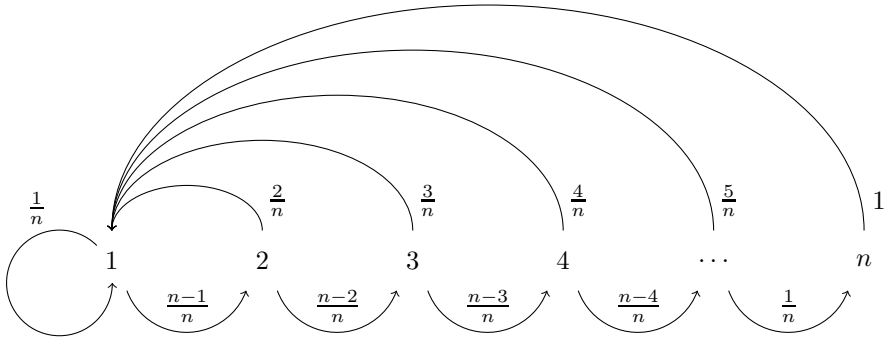
\noindent \begin{tikzpicture}[scale=2]
\node (1) at (0,0) {$1$};
\node (2) at (1,0) {$2$};
\node (3) at (2,0) {$3$};
\node (4) at (3,0) {$4$};
\node (...) at (4,0) {$\ldots$};
\node (5) at (5,0) {$n$};
\draw [->] (0.1,-0.2) arc (200:340:12pt);
\draw [->] (1.1,-0.2) arc (200:340:12pt);
\draw [->] (2.1,-0.2) arc (200:340:12pt);
\draw [->] (3.1,-0.2) arc (200:340:12pt);
\draw [->] (4.1,-0.2) arc (200:340:12pt);
\draw [->] (-0.1,0.1) arc (45:360:10pt);

\draw[<-] (0,0.2) arc(180:0:0.5cm and 0.3cm);
\draw[<-] (0,0.2) arc(180:0:1cm and 0.6cm);
\draw[<-] (0,0.2) arc(180:0:1.5cm and 0.9cm);
\draw[<-] (0,0.2) arc(180:0:2cm and 1.2cm);
\draw[<-] (0,0.2) arc(180:0:2.5cm and 1.5cm);

\node at (-0.5,0.35) {$\frac{1}{n}$};
\node at (1.1,0.4) {$\frac{2}{n}$};
\node at (2.1,0.4) {$\frac{3}{n}$};
\node at (3.1,0.4) {$\frac{4}{n}$};
\node at (4.1,0.4) {$\frac{5}{n}$};
\node at (5.1,0.4) {$1$};
\node at (0.5,-0.3) {$\frac{n-1}{n}$};
\node at (1.5,-0.3) {$\frac{n-2}{n}$};
\node at (2.5,-0.3) {$\frac{n-3}{n}$};
\node at (3.5,-0.3) {$\frac{n-4}{n}$};
\node at (4.5,-0.3) {$\frac{1}{n}$};
\end{tikzpicture}~\\~\\

\noindent Let $\Pi$ denote the unique invariant measure. Then, for $k=2,\ldots,n$, we know 
\[
\Pi(k) = \frac{n-(k-1)}{n}~\Pi(k-1) = \Pi(1) \prod\limits_{l=1}^{k-1} \frac{n-l}{n},
\]
and using $\Pi(1) + \cdots + \Pi(n) = 1$ we find
\[
\Pi(1) = \left( \sum\limits_{k=1}^n \prod\limits_{l=1}^{k-1} \frac{n-l}{n} \right)^{-1}.
\]
Observe that Birkhoff's ergodic theorem yields
\[ \tilde{\gamma} ( \mathcal{K}_n) = n \sum\limits_{k=1}^n \Pi(k) \frac{k}{n} = \sum\limits_{k=1}^n \Pi(k) k.   \]
For small $n$ one can calculate $\Pi$ and the exact value of $\tilde{\gamma} ( \mathcal{K}_n)$ in this way.
\begin{figure}[!h]
\begin{center}
\begin{tabular}{|M{4em}|M{4em}|M{4em}|M{4em}|M{4em}|M{4em}|}
\hline
\multirow{2}{*}{$n$}&\multirow{2}{*}{$1$}& \multirow{2}{*}{$2$}& \multirow{2}{*}{$3$}& \multirow{2}{*}{$4$}& \multirow{2}{*}{$5$}\\ &&&&&\\ \hline
\multirow{2}{*}{$ \tilde{\gamma}(\mathcal{K}_n)$}& \multirow{2}{*}{$1$}& \multirow{2}{*}{$\frac{4}{3}$}& \multirow{2}{*}{$\frac{27}{17}$}& \multirow{2}{*}{$\frac{128}{71}$}& \multirow{2}{*}{$\frac{3125}{1569}$}\\ &&&&& \\ \hline
\end{tabular} \end{center} \end{figure}

\noindent For general $n \geq 1$ Mathematica gives the representation
\begin{equation}
\tilde{\gamma} ( \mathcal{K}_n) = \frac{ n}{ e^n~n^{-n}~\Gamma(n+1,n) - 1},
\end{equation}
A series expansion due to Tricomi, see e.g. Chapter 2.3 in \cite{Gau98}, yields 
\[
\Gamma(n+1,n) \sim \frac{1}{2}~\Gamma(n+1) \quad \textnormal{for}~ n \rightarrow \infty,
\]
and hence allows us to verify
\[
e^n~n^{-n}~\Gamma(n+1,n) \sim e^n~n^{-n}~\frac{1}{2}~\Gamma(n+1) \sim \sqrt{\frac{\pi}{2}}~ n^{1/2} \quad \textnormal{for}~ n \rightarrow \infty.
\]
Applying this to (18) directly gives
\[ \tilde{\gamma} ( \mathcal{K}_n) \sim \sqrt{ \frac{2}{\pi}  }~n^{1/2} \quad \textnormal{for~} n \rightarrow \infty.  \]~\\

\noindent \textbf{Acknowledgement.} The author thanks Elmar Teufl and Martin Zerner for many helpful discussions.

\nocite{BSR13}

\nocite{Sep00}

\bibliography{mybib.bib}
\bibliographystyle{plain}~\\

\noindent \textsc{Eberhard Karls Universit{\"a}t T{\"u}bingen, Mathematisches Insitut,\\
Auf der Morgenstelle 10, 72076 T{\"u}bingen, Germany}\\

\noindent \texttt{georg.braun@uni-tuebingen.de}

\end{document}